\documentclass[letter,10pt]{article}
\usepackage{amssymb}
\usepackage{amsthm}
\usepackage{amsmath}
\usepackage{mathrsfs}
\usepackage{verbatim}

\title{Regularity properties of optimal transportation problems arising in hedonic pricing models}
\author{Brendan Pass\footnote{Department of Mathematical and Statistical Sciences, 632 CAB, University of Alberta, Edmonton, Alberta, Canada, T6G 2G1 pass@ualberta.ca.}}

\begin{document}

\maketitle

\begin{abstract}
We study a form of optimal transportation surplus functions which arise in hedonic pricing models.  We derive a formula for the Ma-Trudinger-Wang curvature of these functions, yielding necessary and sufficient conditions for them to satisfy \textbf{(A3w)}.  We use this to give explicit new examples of surplus functions satisfying \textbf{(A3w)}, of the form $b(x,y)=H(x+y)$ where $H$ is a convex function on $\mathbb{R}^n$.  We also show that the space of equilibrium contracts in the hedonic pricing model has the maximal possible dimension, a result of potential economic interest.
\end{abstract}

\section{Introduction}
Fix Borel probability measures $\mu$ and $\nu$ on smooth manifolds $X$ and $Y$, respectively, and a smooth surplus function $b: X \times Y \rightarrow \mathbb{R}$.  Monge's optimal transportation problem is to find the map $F: X \rightarrow Y$, pushing $\mu$ forward to $\nu$, which maximizes the transportation surplus:

\begin{equation*}
\int_{X}b(x,F(x))d\mu.
\end{equation*}

Assuming some regularity on $\mu$ (say absolute continuity with respect to local coordinates), sufficient conditions on $b$ for the existence and uniqueness of the minimizer $F$ were found by Levin \cite{lev}, building on results of Gangbo and McCann \cite{gm}, Caffarelli \cite{Caf} and Gangbo \cite{g}.  The regularity, or smoothness, of $F$ is currently a very hot area of research.  Ma, Trudinger and Wang \cite{mtw} found a condition \textbf{(A3s)} on $b$ under which the optimizer must be smooth as long as the marginals $\mu$ and $\nu$ are smooth and bounded above and below.  Trudinger and Wang  then proved that a weaker version of this hypothesis, \textbf{(A3w)}, is in fact sufficient for the regularity of $F$ \cite{tw1}\cite{tw2} and Loeper showed that, even for rougher marginals, \textbf{(A3w)} is both necessary and sufficient for the continuity of $F$ \cite{loeper}.  This framework generalizes and unifies a series of earlier regularity results obtained by Caffarelli \cite{c1}\cite{c2}\cite{c2}, Urbas \cite{u}, Delanoe \cite{d1}\cite{d2} and Wang \cite{wang}.  Since then, many interesing results on regularity have been obtained; see, for example, \cite{liu}\cite{km}\cite{km2}\cite{loeper2}\cite{lv}\cite{fr}\cite{frv}\cite{frv2}\cite{fkm3}\cite{fkm2}\cite{frv3}\cite{k}. An interesting line of current research is to find examples of surplus functions satisfying \textbf{(A3w)} and \textbf{(A3s)}.

One main goal in the present paper is to determine when \textbf{(A3w)} holds for a particular class of surplus functions.  Explicitly, we are interested in functions of the form:

\begin{equation}\label{form}
 b(x,y) = \sup_{z \in Z} h(x,z) + g(y,z)
\end{equation}

Our motivation in studying surplus functions of this form comes partially from mathematical economics.  A recent paper by Chiappori, McCann and Nesheim \cite{cmn} demonstrated that finding an equilibrium in a certain hedonic pricing model is equivalent to solving an optimal transportation problem with a surplus function of this form (see also Ekeland \cite{e} \cite{e2} and Carlier and Ekeland \cite{ce} for another approach to this problem).  We briefy review this model now.  Imagine $X$ parameterizes a population of consumer types who are looking to buy some specific good (say houses).  The different directions in the manifold $X$ may represent various characteristics which differentiate among types (for example, age, income, family size, etc) and $d\mu(x)$ represents the relative frequency of types with characteristics $x \in X$.  Suppose now that $Y$ parameterizes a space of sellers looking to produce and sell the same good (say companies looking to build customized houses).  Again, the different directions in $Y$ represent characteristics differentiating seller types from each other (for example, the size of the company and the location of its headquarters) and $d\nu(y)$ their relative frequencies.  Suppose that $Z$ represents the space of available goods that can potentially be produced (for example, the collection of all houses that can feasibly be built, characterized by their size, location, etc.).  The space $Z$ is often referred to as the space of contracts. The functions $h(x,z)$ and $g(y,z)$ represent the preference of consumer $x$ to buy house $z$ and the preference of seller $y$ to build house $z$, respectively.  The result of Chiappori, McCann and Nesheim implies that the equilibrium coupling of buyers to sellers is also a solution to the optimal transportation problem with marginals $\mu$ and $\nu$ and surplus $b(x,y)$ given by equation (\ref{form}).  Despite their relevance in mathemtical economics, however, optimal transportation problems with surplus functions of this form do not seem to have been studied systematically in the literature. 

The \textbf{(A3w)} condition is that a certain tensor -- the Ma-Trudinger-Wang curvature, defined in the next section -- should be nonnegative on a certain set of vectors.  Our goal here is to find a formula for the Ma-Trudinger-Wang curvature of $b$ in terms of $g$ and $h$.  This will then yield necessary and sufficient conditions on $g$ and $h$ in order for $b$ to satisfy \textbf{(A3w)}. 

In particular, our work here will produce sufficient conditions on $g$ and $h$ which will ensure that the equilibrium assignment of sellers to buyers in the hedonic pricing problem is continuous.  In addition, this result should be of independent interest to mathematicians working in optimal transportation, as at present there are few known examples of surplus functions satisfying \textbf{(A3w)}.  Our result here yields new examples of such functions.  In particular, we show that for special choices of the preference functions $h$ and $g$, we obtain $b(x,y) = H(x+y)$, where $H$ is a convex function on $\mathbb{R}^n$.\footnote{Surplus functions of this form are reminiscent of a class of functions studied by Gangbo and McCann \cite{gm}, who were interested in minimizing the transportation cost for cost functions of the form $l(|x-y|)$, where $l \geq 0$ is strictly concave.}  Our work here yields a new formula for the Ma-Trudinger-Wang curvature in terms of the Legendre dual $H^*$ of $H$; this formula splits into two terms, one of which is always positive.

In a recent paper by Figalli, Kim and McCann, a variant of \textbf{(A3w)}, called \textbf{(B3w)}, emerged as a central concept in another type of economic problem \cite{fkm}.   This condition also asserts the positivity of Ma-Trudinger-Wang curvature, but on a larger set of vectors than \textbf{(A3w)}.  Let us mention here that, as we derive a formula for the Ma-Trudinger-Wang curvature, our method also yields necessary and sufficient conditions on $g$ and $h$ so that $b$ satisfies \textbf{(B3w)}.

In the next section, we recall various structural hypotheses of the function $b$ which arise in the regularity theory of optimal transportation, \textbf{(A3w)} and \textbf{(A3s)} being the most important.  In section 3, we derive conditions on $g$ and $h$ that ensure $b$ will satisfy these hypotheses.  In the final section, we study the structure of the set of contracts that gets signed in equilibrium; explicity, we prove that the dimension of this set is the same as the dimension of the support of the optimal measure.

It is a pleasure to thank Robert McCann for suggesting this problem to me and for useful discussions during the course of this work.
\section{Assumptions and defintions }

We will assume that the domains $X$, $Y$ and $Z$ can all be smoothly imbedded in larger manifolds, such that their closures $\overline{X}$, $\overline{Y}$ and $\overline{Z}$ are compact.  We will also make several assumptions on the preference functions $g$ and $h$.

\begin{enumerate}

\item The smooth manifolds $X$, $Y$, and $Z$ all share the same dimension, which we will denote by $n$.
\item $h \in C^{2}(\overline{X} \times \overline{Z})$ and $g \in C^{2}(\overline{Y} \times \overline{Z})$.
\item For each $(x,y)$ the supremum is attained by a unique $z(x,y) \in Z$.
\item For all $(x,y)$, the $n \times n$ matrix $D^2_{zz}h(x, z(x,y)) + D^2_{zz}g(y, z(x,y))$ is non-singular.

\end{enumerate}

Economically, $z(x,y)$ is the contract that maximizes the total utility of agents $x$ and $y$; if, in equilibrium, $x$ and $y$ choose to conduct business with each other, $z(x,y)$ is the contract they sign.

Assumption (1) on the dimensions is not entirely necessary.  Most of the literature on optimal transportation deals with equal dimensional spaces; one exception is a recent paper by the present author on the regularity of optimal transportation when the dimensions differ \cite{P}.  For economic applications, however, it may be desirable to allow these dimensions to differ.  The dimensions of $X$, $Y$ and $Z$ may represent the number of characteristics used in the model to differentiate among buyers, sellers and contracts, respectively, and it is certainly possible that these will not coincide.  For simplicity, we assume here that the dimensions are all equal, but note that it is straghtforward to extend the analysis to the case $\dim(X) \geq \dim(Z) \geq \dim(Y)$, using the extensions of the conditions \textbf{(A0)}-\textbf{(A3s)} found in \cite{P}.    
   
\newtheorem{twist}{Definition}[section]
\begin{twist}
We say $b$ is $(x,y)$-twisted if for all $x \in \overline{X}$, the mapping $y \mapsto D_xb(x,y)$ is injective on $\overline{Y}$.  We say $b$ is $(y,x)$-twisted if for all $y \in \overline{Y}$, the mapping $x \mapsto D_yb(x,y)$ is injective on $\overline{X}$.  We say $b$ is bi-twisted if it is both $(x,y)$ and $(y,x)$-twisted.
\end{twist}
 If $b$ is $(x,y)$-twisted, the map $y \mapsto D_xb(x,y)$ is invertible on its range and we will denote its inverse by $b$-exp$_{x}(\cdot)$.

\newtheorem{nondeg}[twist]{Definition}
\begin{nondeg}
We say $b$ is non-degenerate if for all $(x,y)$ the matrix of mixed, second order partials $D_{xy}^2b(x,y)$ is non-singular. 
\end{nondeg}

We will use analogous terminology for $g$ and $h$; for example, we will say that $g$ is $(y,z)$-twisted if $z \mapsto D_yg(y,z)$ is injective. 

The first three regularity conditions formulated by Ma, Trudinger and Wang are:

\begin{description}
	\item[(A0)] $b \in C^4(\overline{X} \times \overline{Y})$.
	\item[(A1)] $b$ is bi-twisted.
	\item[(A2)] $b$ is non-degenerate. 
\end{description}
 
For sufficiently regular marginals $\mu$ and $\nu$, \textbf{(A1)} implies the existence and uniqueness of a maximizer $F$, \cite{lev}\cite{g}\cite{gm}\cite{Caf}.  The condition \textbf{(A2)}, in turn, implies that the graph of $F$ is contained in an $n$-dimensional Lipschitz submanifold of the product $X \times Y$ \cite{mpw}.    

Our next definition concerns the structure of the domain $Y.$
\newtheorem{bcon}[twist]{Definition}
\begin{bcon}
We say $Y$ is $b$-convex if for all $x$ the set $D_xb(x,Y) \subseteq T_x^*X$ is convex. 
\end{bcon}

Ma, Trudinger and Wang showed that the $b$-convexity of $Y$ is necessary for the continuity of the map $F$ for arbitrary smooth marginals $\mu$ and $\nu$ \cite{mtw}.  Assuming this condition, as well as, \textbf{(A0)}-\textbf{(A2)} they showed that under, \textbf{(A3s)}, which we define below, the optimal map $F$ is smooth.  Loeper then showed that the weakening \textbf{(A3w)} of this condition is necessary and sufficient for the continuity of $F$ \cite{loeper}.

To formulate \textbf{(A3w)} and \textbf{(A3s)}, we will need the following definition.

\newtheorem{curv}[twist]{Definition}
\begin{curv}
 Assume \textbf{(A0)}-\textbf{(A2)} hold and that $Y$ is $b$-convex.  Let $x \in X$ and $y \in Y$.  Choose tangent vectors $v\in T_x X$ and $u \in T_yY$.  Set $q=D_xb(x,y) \in T_x^{*}X$ and $p = (D^2_{xy}b(x,y))\cdot u \in T_x^{*}X$.  For any smooth curve $\beta(s)$ in $X$ with $\beta(0)=x$ and $\frac{d\beta}{ds}(0)=v$, we define the Ma-Trudinger-Wang curvature of $b$ at $x$ and $y$, in the directions $v$ and $u$ by:
\begin{equation*}
 MTW^b_{(x,y)}\langle v,u\rangle : =\frac{\partial^4}{\partial s^2\partial t^2}b(\beta(s),b\text{-}exp_x(tp+q))
\end{equation*}
\end{curv}

A local coordinates expression for $MTW^b_{(x,y)}\langle v,u\rangle$ was first introduced by Ma, Trudinger and Wang \cite{mtw}.  The formulation above is due to Loeper \cite{loeper}, who showed that $MTW^b_{(x,y)}\langle v,u\rangle$ is invariant under smooth changes of coordinates and, when, $-b(x,y)$ is the quadratic cost on a Riemannian manifold, it is equal to the sectional curvature along the diagonal.  For general smooth surplus functions, Kim and McCann \cite{km} showed that it is the sectional curvature of certain null planes, corresponding to a certain pseudo-Riemannian metric.

We can now state the final regularity conditions:

\begin{description}
	\item[(A3w)] For all $(x,y) \in X \times Y$, $v\in T_x X$ and $u \in T_yY$ such that $v^T\cdot D^2_{xy}b\cdot u=0$ we have $MTW^b_{(x,y)}\langle v,u\rangle \geq 0.$
	\item[(A3s)] For all $(x,y) \in X \times Y$, $v\in T_x X$ and $u \in T_yY$ such that $v^T\cdot D^2_{xy}b\cdot u=0$ and $v, u \neq 0$ we have $MTW^b_{(x,y)}\langle v,u\rangle > 0.$
	
\end{description}
In subsequent sections, we will often refer to the curve $t \mapsto b\text{-}exp_x(tp+q) \in Y$ as a\textit{ $b$-segment} in $Y$.  We also note here that the conditions \textbf{(B3w)} and \textbf{(B3s)} found in \cite{fkm} are equivalent to \textbf{(A3w)} and \textbf{(A3s)}, respectively, without the orthogonality condition $v^T\cdot D^2_{xy}b\cdot u=0$.
\section{Regularity properties of $b$}
The aim of this section is to understand when $b$ satsifies certain regularity properties, namely \textbf{(A1)}, which ensures the existence and uniqueness of the optimal map $F$, \textbf{(A2)}, ensuring the rectifiability of the graph of $F$, $b$-convexity of $Y$ and \textbf{(A3w)}/\textbf{(A3s)}, governing the regularity of the optimal map.

First, we verify some simple facts.

\newtheorem{facts}{Lemma}[section]
\begin{facts}\label{facts}
 The map $z(x,y)$ is continuously differentiable and $b$ is $C^2$ smooth.  For all $x,y$ we have the following formulae:
\begin{eqnarray}
D_xb(x,y) &=& D_xh(x,z(x,y)) \label{b_x}\\
D_yb(x,y) &=& D_yg(y,z(x,y))\label{b_y} \\
 D_xz(x, y)&=&-M^{-1}(x, y)D^2_{zx}h(x,z) \label{z_x}\\
D_yz(x, y)&=&-M^{-1}(x, y)D^2_{zy}g(y,z)\label{z_y}\\
D^2_{xy}b(x,y)&=&-D^2_{xz}h(x,z)M^{-1}(x, y)D^2_{zy}g(y,z)\label{b_xy}
\end{eqnarray}
where $M(x,y):=D^2_{zz}h(x,z(x,y))+D^2_{zz}g(y,z(x,y))$
\end{facts}

\begin{proof}
Note that, as $z(x,y)$ maximizes $z \mapsto g(x,z) + h(y,z)$, we have 

\begin{equation}
D_zg(x,z(x,y)) + D_zh(y,z(x,y)) = 0 \label{D_z}
\end{equation}
As $D^2_{zz}g(x, z(x,y)) +D^2_{zz}h(y, z(x,y))$ is non-singular by assumption, the Implicit Function Theorem now implies that $z(x,y)$ is $C^1$.  This, in turn, implies that $b(x,y) = g(x,z(x,y)) + h(y,z(x,y))$ is at least $C^1$.  Now note that $b(x,y) - g(x,z) +h(y,z) \geq 0$, for all $z$, with equality when $z=z(x,y)$, which implies $(\ref{b_x})$ and $(\ref{b_y})$.  Differentiating $(\ref{D_z})$ with respect to $x$ and $y$, respectively, yields $(\ref{z_x})$ and $(\ref{z_y})$.  Differentiating ($\ref{b_x}$) with respect to $y$ and using ($\ref{z_y}$) yields ($\ref{b_xy}$).
\end{proof}

We will first prove results about $\textbf{(A0)} - \textbf{(A2)}$.  These proofs are based on similar arguments, found in \cite{cmn} and \cite{P}; we recreate them here for the reader's convenience.

\newtheorem{smooth}[facts]{Corollary}
\begin{smooth}
If $h$ and $g$ satisfy $\textbf{(A0)}$, then $b$ satisfies $\textbf{(A0)}$ as well.
\end{smooth}
\begin{proof}
Using (\ref{D_z}) and the Implicit Function Theorem we find that $z(x,y)$ is $C^3$; equations (\ref{b_x}) and (\ref{b_y}) together with the chain rule now imply the desired result. 
\end{proof}

\newtheorem{ngt}[facts]{Proposition}
\begin{ngt}
If $h$ and $g$ are non-degenerate, then $b$ is non-degenerate.  If $h$ is $(x,z)$-twisted and $g$ is $(z,y)$-twisted, then $b$ is $(x,y)$-twisted.
\end{ngt}

\begin{proof}
The non-degeneracy part of the proposition follows immediately from $(\ref{b_xy})$.  

Assume that $h$ is $(x,z)$ twisted and $g$ is $(z,y)$ twisted.  Suppose we have $D_xb(x,y_0) = D_xb(x,y_1)$; we need to show $y_0 =y_1$.  By $(\ref{b_x})$, we have $D_xg(x,z(x,y_0)) = D_xg(x,z(x,y_1))$ and so by the twistedness of $g$ we have $z(x,y_0) =z(x,y_1)$.  Now, for $i=0,1$ we have $D_zg(x,z(x,y_i)) + D_zh(y_i,z(x,y_i)) = 0$ by \ref{D_z}.  Therefore, as $z(x,y_0) =z(x,y_1)$,
\begin{equation*}
D_zh(y_0,z(x,y_0)) = -D_zg(x,z(x,y_0)) = - D_zg(x,z(x,y_1))=D_zh(y_1,z(x,y_1))
\end{equation*} 

Again using $z(x,y_0) =z(x,y_1)$, the equality $D_zh(y_0,z(x,y_0)) = D_zh(y_1,z(x,y_1))$ and the twistedness of $h$ imply that $y_0 = y_1$ as desired.
\end{proof}

Of course, an analagous result holds if $h$ is $(z,x)$-twisted and $g$ is $(y,z)$-twisted and so we immediately obtain the following.

\newtheorem{a1a2}[facts]{Corollary}
\begin{a1a2}
If $h$ and $g$ satsify $\textbf{(A1)}$, then so does $b$.  If $h$ and $g$ satisfy $\textbf{(A2)}$, then so does $b$.
\end{a1a2}

Next, we consider the $b$-convexity condition.

\newtheorem{bcon2}[facts]{Proposition}
\begin{bcon2}
Assume $Z$ is $h$-convex and for all $h$-segments $z_t$ at $x$,  $-D_zh(x,z_t)$ is in the domain of $g-exp_{z_t}(\cdot)$.  Then the domain $Y$ is $b$-convex.
\end{bcon2}
 \begin{proof}
Let $D_xb(x,y_i)=p_i$, for $i=0,1$.  For all $t \in [0,1],$ we must show that there is some $y_t \in Y$ such that $D_xb(x,y_t)=t(p_1-p_0)+p_0:=p_t$.  Let $z_i=z(x,y_i)$; then $D_xh(x,z_i)=D_xb(x,y_i)=p_i$ for $i=0,1$, by \ref{b_x} in Lemma \ref{facts}.  The given conditions imply the existence of a $z_t$ such that $D_xh(x,z_t)=p_t$ and a $y_t$ such that $D_zh(x,z_t)+D_zg(y_t,z_t)=0$.  Therefore, $z_t =z(x,y_t)$, and so $D_xb(x,y_t)=D_xh(x,z_t)=p_t$, as desired.
 \end{proof}
 
We now provide a formula for the Ma-Trudinger-Wang curvature in terms of $g$ and $h$.
\newtheorem{mtw}[facts]{Theorem}
\begin{mtw}\label{mtwgh}
 Let $x \in X$, $z \in Z$, $v \in T_xX$ and $z_t$ be a $b$-segment at $x$.  Set $y_t =g$-$exp_{z_t}(-D_zh(x,z_t))$.  Then, for $u=\dot{y_t}$, $MTW^b_{(x,y_0)}\langle v,u\rangle$  is given by

\begin{equation}\label{cc}
\frac{d^2}{dt^2}\Big|_{t=0} v^T\cdot \Big( D^2_{xx}h(x,z_t) - D^2_{xz}h(x,z_t) \cdot \big( D^2_{zz}h(x,z_t) + D^2_{zz}g(y_t,z_t) \big )^{-1}\cdot D^2_{zx}h(x,z_t) \Big) \cdot v
\end{equation}
\end{mtw}

To prove this result, we will need a simple lemma about the Ma-Trudinger-Wang curvature.  This lemma is well known, and the resulting formula for the Ma-Trudinger-Wang curvature appears in, for example, \cite{k}, but it does not seem to be proven explicitly in the literature, so we provide a proof below.

\newtheorem{crosscurv}[facts]{Lemma}
\begin{crosscurv}
Let $y_t$ be a $b$-segment at $x$: $D_xb(x,y_t) =tp+q$.  Then, for $u=\dot{y_t}$,
\begin{equation*}
MTW^b_{(x,y_0)}\langle v,u \rangle =\frac{d^2}{dt^2}\Big|_{t=0} v^T\cdot D^2_{xx}b(x,y_t) \cdot v
\end{equation*}  
\end{crosscurv}
\begin{proof}
Letting $x_s$ be a curve such that $x_0 =x$ and $\dot{x_0} = v$, we have

\begin{eqnarray*}
MTW^b_{(x,y_0)}\langle v,u \rangle &=& \frac{d^4}{dt^2ds^2}\Big|_{t=0,s=0} b(x_s,y_t)\\
&=&\frac{d^3}{dt^2ds}\Big|_{t=0,s=0} D_xb(x_s,y_t)\cdot \dot{x_s}\\
&=&\frac{d^2}{dt^2}\Big|_{t=0,s=0} (\dot{x_s}^T\cdot D^2_{xx}b(x_s,y_t)\cdot \dot{x_s}+D_xb(x_s,y_t)\cdot \ddot{x_s})\\
&=&\frac{d^2}{dt^2}\Big|_{t=0} (v^T\cdot D^2_{xx}b(x,y_t)\cdot v)+\frac{d^2}{dt^2}\Big|_{t=0}(D_xb(x,y_t)\cdot\ddot{x_0})\\
&=&\frac{d^2}{dt^2}\Big|_{t=0} (v^T\cdot D^2_{xx}b(x,y_t)\cdot v)+\frac{d^2}{dt^2}\Big|_{t=0}((pt+q)\cdot\ddot{x_0})\\
&=&\frac{d^2}{dt^2}\Big|_{t=0} (v^T\cdot D^2_{xx}b(x,y_t)\cdot v)
\end{eqnarray*}

\end{proof}

We can now prove Proposition \ref{mtwgh}.

\begin{proof}
In light of the preceding lemma, it suffices to show:

\begin{equation*}
D^2_{xx}b(x,y_t) =  D^2_{xx}h(x,z_t) - D^2_{xz}h(x,z_t) \cdot \big( D^2_{zz}h(x,z_t) + D^2_{zz}g(y_t,z_t) \big )^{-1}\cdot D^2_{zx}h(x,z_t) \end{equation*}
By Lemma \ref{facts},
\begin{equation*}
D_xb(x,y) = D_xh(x,z(x,y))
\end{equation*}

Differentiating this equation with respect to $x$, noting that $z(x,y_t) =z_t$ and using the formula in Lemma \ref{facts} for $D_xz(x,y)$ yields the desired result.
\end{proof}

Equation (\ref{cc}) for the Ma-Trudinger-Wang curvature naturally splits into two terms:

\begin{equation*}
MTW^b_{(x,y_0)}\langle v,u \rangle = A + B
\end{equation*}

where 

\begin{equation*}
A = \frac{d^2}{dt^2}\Big|_{t=0} v^T \cdot \Big( D^2_{xx}h(x,z_t)\Big) \cdot v, \text{  } B =-\frac{d^2}{dt^2}\Big|_{t=0} v_t^T\cdot(M_t)^{-1}\cdot v_t
\end{equation*}
Here $M_t = D^2_{zz}h(x,z_t) + D^2_{zz}g(y_t,z_t)$, and  $v_t = D^2_{zx}h(x,z_t) \cdot v$.  The first term, $A$, is exactly the Ma-Trudinger-Wang curvature of $h$.  The second term $B$, can be further refined using the chain rule and the formulae for the derivatives of the matrix $M_t^{-1}$:
 
 \begin{equation*}
 \dot{(M_t^{-1})} = -M_t^{-1}\dot{M_t}M_t^{-1}, \text{ } \ddot{(M_t^{-1})} = -M_t^{-1}\ddot{M_t}M_t^{-1} + 2M_t^{-1}\dot{M_t}M_t^{-1}\dot{M_t}M_t^{-1} 
 \end{equation*}
Differentiating $v_t^T\cdot(M_t)^{-1}\cdot v_t$ twice, applying these formulae and using the symmetry of $M_t$ and its derivatives yields:
 
 \begin{eqnarray}\label{B}
 B &=&\Big(-2 \ddot{v_t}^T \cdot M_t^{-1}\cdot v_t - 2 \dot{v_t}^T\cdot M_t^{-1}\cdot\dot{v_t} +4 \dot{v_t}^T\cdot M_t^{-1}\dot{M_t}M_t^{-1}\cdot v_t +  \nonumber \\ 
 &+&v_t^T\cdot M_t^{-1}\ddot{M_t}M_t^{-1}\cdot v_t  -  2v_t^T\cdot M_t^{-1}\dot{M_t}M_t^{-1}\dot{M_t}M_t^{-1}\cdot v_t\Big)|_{t=0}
 \end{eqnarray}

Now, note that the maximality of $z \mapsto h(x_t,z) + g(y_t,z)$ at $z_t= z(x,y_t)$ implies that $M_t$, the second derivative of this map, is negative semi-definite.  Assumption 4 at the beginning of section 2 asserts that $M_t$ is non-singular, and therefore it is negative-definite.  The second and last terms in (\ref{B}) are therefore non-negative, due to the negative definiteness of $M_t^{-1}$ and the symmetry of $M_t^{-1}$ and $\dot{M_t}$.

As the following examples shows, for appropriate forms of the functions $h$ and $g$, $b(x,y)=H(x+y)$ becomes an arbitrary convex function of the sum.  Understanding when functions of this form have positive cross curvature is interesting in its own right.  In this setting, the present approach --- essentially doing calculations on $H^*$ rather than $H$ --- has a distinct advantage; $b$-segments for $H$ correspond to ordinary line segments for the dual variables.  That is, instead of evaluating $H$ along a $b$-segment, we evaluate $H^*$ along a line.  

\newtheorem{ex2}[facts]{Example}
\begin{ex2}
Take $X=Y=Z =\mathbb{R}^n$ and set $h(x,z) =x\cdot z -H^*(z)$, where $H^*$ is the Legendre transform of some smooth, uniformly convex function $H:\mathbb{R}^n \rightarrow \mathbb{R}$, and $g(y,z)=y \cdot z$.  Then $b(x,y)=H(x+y)$.  Then $MTW^b_{(x,y)}(v,u)$ is given by:

\begin{equation}
\sum_{i,j,k,l}-\frac{\partial^4 H^*}{\partial z^i\partial z^j\partial z^k\partial z^l} p_k p_lw_i w_j+2\sum_{i,j,k,l,a,r}\frac{\partial^3 H^*}{\partial z^i\partial z^l\partial z^k}\frac{\partial^3 H^*}{\partial z^r\partial z^j\partial z^a}\frac{\partial^2 H}{\partial z^l\partial z^r}w_i w_j p_a p_k  \label{mtwH}
\end{equation}
where $p=D^2H(x+y)\cdot u$, $w = D^2H(x+y)\cdot v$.
\end{ex2}

\begin{proof}

The curve $y_t$ is given by $DH(x+y_t)=tp+q$, or $y_t = DH^*(tp+q) - x$, and $z_t$ is given by $z_t=D_xh(x,z_t) = tp +q$.
Note that $A =0$, as $D_{xx}^2h(x,z_t) =0$.  Turning to $B$, note that as $D^2_{xz}h(x,z_t) =I$, we have $v_t=v$, and $M_t =-D^2H^*(tp+q)$, so that $\dot{M_t} = -D^3H^*(tp+q)\cdot p$, or, in matrix notation,  $(\dot{M_t})_{ij} =-\sum_{k}\frac{\partial^3 H^*}{\partial z^i\partial z^j\partial z^k}p_k$.  Similarly, $(\ddot{M_t})_{ij} = -\sum_{k,l}\frac{\partial^4 H^*}{\partial z^i\partial z^j\partial z^k\partial z^l} p_k p_l$.  

Now, the first three terms in (\ref{B}) vanish. The final two terms are exactly the desired expression. 
\end{proof}
The formula (\ref{mtwH}) for the Ma-Trudinger-Wang curvature does not seem much simpler than the original formula of Ma, Trudinger and Wang \cite{mtw}, but it has one important advantage; the second term is always non-negative, because of the symmetry of mixed partials and the positive definitness of $D^2H$.  We obtain immediately, for example, that on any domain where the fourth order derivatives of $H^*$ are small compared to its third order derivatives and the second order derivatves of $H$, $b$ satisfies \textbf{(A3s)} (in fact, it satisfies the stronger condition \textbf{(B3s)}).

\section{Hausdorff dimension of the space of contracts}
Here we study the structure of the distribution of contracts that get signed in equilibrium in the hedonic pricing problem.  As in previous sections, we will assume for simplicity that $\dim(X) = \dim(Y) =\dim(Z) :=n$, but remark that the results may be extended to the case where   $\dim(X) \geq \dim(Z) \geq \dim(Y)$ in a straightforward way.

Assuming that $\mu$ assigns zero mass to every set of Hausdorff dimension less than or equal to $n-1$, the twist condition implies the existence of a unique map $F: X \rightarrow Y$ solving Monge's optimal transportation problem; as $b$ is $C^2$, this map is differentiable almost everywhere.  In economic terms, this means that, in equilibrium almost every agent $x$ conducts business with a unique agent $y := F(x)$; they sign the contract $z(x,F(x))$.  We define the distribution of signed contracts $\mu_Z$ to be the push forward of $\mu$ by the map $x \mapsto z(x,F(x))$;  our goal here is to investigate the structure of this measure.   Economically, the support of this measure can be interpreted as the set of all contracts that are executed in equilibrium, while $d\mu_Z(z)$ represents the relative frequency of contracts that are executed in equilibrium.  Our result below implies that this set has the maximal possible dimension.

\newtheorem{invert}{Proposition}[section]
\begin{invert}
Assume that $h$ is $(x,z)$-twisted, $g$ is $(z,y)$-twisted and that both $g$ and $h$ are non-degenerate.  At any point where the map $x \mapsto z(x,F(x))$ is differentiable, it's derivative has full rank.
\end{invert}

\begin{proof}
 Fix a point $x_0$ where $F$ is well defined and differentiable.  Set $y_0=F(x_0)$ and $z_0=z(x_0,y_0)$.  It is well known that there is a function $u:X \rightarrow \mathbb{R}$, twice differentiable almost everywhere, such that for all $x$ where $u$ is differentiable we have
\begin{equation*}
 Du(x) = D_xb(x, F(x))
\end{equation*}
Wherever $F$ is differentiable, $u$ is twice differentiable, and we have
\begin{equation}\label{potential}
 D^2u(x) = D^2_{xx}b(x, F(x))+D^2_{xy}b(x,F(x))DF(x)
\end{equation}
In particular, the preceding equality holds at $x =x_0$.  Set $P_0=D^2u(x_0) - D^2_{xx}b(x_0, y_0)$; it is well known that $P_0 \geq 0$.  Rearranging (\ref{potential}) implies

\begin{equation} \label{df}
 DF(x_0) = \big(D^2_{xy}b(x_0,y_0)\big)^{-1}P_0.
\end{equation}

Now, the derivative of $x \mapsto z(x,F(x))$ at $x_0$ is

\begin{equation*}
 D_xz(x_0, y_0) + D_yz(x_0, y_0)DF(x_0)
\end{equation*}
 
Now, using the formulae in Lemma \ref{facts} together with (\ref{df}) we see that this is equal to
\begin{eqnarray}
&& -M^{-1}(x_0, y_0)\cdot D^2_{zx}h(x_0,z_0) + \big(D^2_{xz}h(x_0,z_0)\big)^{-1}\cdot P_0 \nonumber \\
&& = \Big(-M^{-1}(x_0, y_0) + \big(D^2_{xz}h(x_0,z_0)\big)^{-1}\cdot P_0\cdot \big(D^2_{zx}h(x_0,z_0)\big)^{-1}\Big)\cdot D^2_{zx}h(x_0,z_0)\nonumber \\ 
&&\label{final}
\end{eqnarray}

Now, as $M(x_0, y_0) <0$, we have $-M^{-1}(x_0, y_0) >0$.  Also, $P_0 \geq 0$, hence 
\begin{equation*}
\big(D^2_{xz}h(x_0,z_0)\big)^{-1}P_0\big(D^2_{zx}h(x_0,z_0)\big)^{-1} = \big(D^2_{xz}h(x_0,z_0)\big)^{-1}P_0\Big(\big(D^2_{xz}h(x_0,z_0)\big)^{-1}\Big)^T \geq 0.  
\end{equation*}
Therefore, the term in brackets in (\ref{final}) is positive definite and thus invertible, which, as $D^2_{zx}h(x_0,z_0)$ is invertible implies the invertibility of (\ref{final}).
\end{proof}

This proposition immdediately implies that a large set of contracts gets signed in equilibrium:

\newtheorem{dimension}{Corollary}[invert]
\begin{dimension}
Assuming that the support of $\mu$ has Hausdorff dimension $n$, the distribution $\mu_Z$ of signed contracts has $n$ dimensional support.
\end{dimension}

If the dimensions of $X$, $Y$, and $Z$ were not the same, but instead we had  $\dim(X) \geq \dim(Z) \geq \dim(Y)$, the proof of the proposition would formally be the same, but we would have to be careful to interpret the inverses of $D^2_{xz}h$ and $D^2_{yz}g$ as one sided inverses (the modification of non-degeneracy condition found in \cite{P} is that these matrices have full rank, so that their one sided inverses exist).  Of course, in this case, the dimension of the support of $\mu_Z$ would be $\dim(Z)$.

\end{document}